\documentclass{amsart}

\usepackage{datetime}

\usepackage[all,dvips]{xy}
\usepackage[dvips]{graphicx}

\usepackage[usenames,dvipsnames]{xcolor}
\usepackage{multirow}
\usepackage[all,dvips]{xy}
\usepackage[dvips]{graphicx}
\usepackage[T1]{fontenc}
\usepackage{longtable}
\usepackage{amssymb,amsmath}
\usepackage{pdflscape}
\usepackage{todonotes} 

\DeclareMathOperator{\tors}{tors}

\begin{document}
\newtheorem{theorem}{Theorem}
\newtheorem{corollary}[theorem]{Corollary}
\newtheorem{conjecture}[theorem]{Conjecture}
\newtheorem{lemma}[theorem]{Lemma}
\newtheorem{proposition}[theorem]{Proposition}
\newtheorem{axiom}{Axiom}[section]
\newtheorem{exercise}{Exercise}[section]
\newtheorem{problem}{Problem}[section]
\newtheorem{definition}[theorem]{Definition}

\newcommand{\cC}{{\mathcal{C}}}
\newcommand{\cO}{{\mathcal{O}}}
\newcommand{\Q}{{\mathbb{Q}}}
\newcommand{\Z}{{\mathbb{Z}}}
\newcommand{\bP}{{\mathbb{P}}}

\title[Torsion of rational elliptic curves over cubic fields]{Torsion of rational elliptic curves\\ over cubic fields}

\author{Enrique Gonz\'alez--Jim\'enez}
\address{Universidad Aut{\'o}noma de Madrid, Departamento de Matem{\'a}ticas and Instituto de Ciencias Matem{\'a}ticas (ICMat), Madrid, Spain}
\email{enrique.gonzalez.jimenez@uam.es}
\urladdr{http://www.uam.es/enrique.gonzalez.jimenez}
\author{Filip Najman}
\address{University of Zagreb, Bijeni\v{c}ka cesta 30, 10000 Zagreb, Croatia}
\email{fnajman@math.hr}
\urladdr{http://web.math.pmf.unizg.hr/~fnajman/}
\author{Jos\'e M. Tornero}
\address{Departamento de \'Algebra, Universidad de Sevilla. P.O. 1160. 41080 Sevilla, Spain.}
\email{tornero@us.es}
\thanks{The first author was partially  supported by the grant MTM2012--35849. The third author was partially supported by the grants FQM--218 and P12--FQM--2696 (FSE and FEDER (EU))}

\date{\today\,\,:\,\,\currenttime}

\begin{abstract}
Let $E$ be an elliptic curve defined over $\Q$. We study the relationship between the torsion subgroup $E(\Q)_{\tors}$ and the torsion subgroup $E(K)_{\tors}$, where $K$ is a cubic number field. In particular, We study the number of cubic number fields $K$ such that $E(\Q)_{\tors}\neq E(K)_{\tors}$.
\end{abstract}

\maketitle

\section{Introduction}
Let $K$ be a number field. The Mordell-Weil Theorem states that the set of $K$-rational points of an elliptic curve $E$ defined over $K$ is a finitely generated abelian group. That is, $E(K) \simeq E(K)_{\tors} \oplus \Z^r$, where  $E(K)_{\tors}$ is the torsion subgroup and $r$ is the rank. Moreover, it is well known that $E(K)_{\tors}\simeq \cC_m\times\cC_n$ for two positive integers $n,m$, where $m$ divides $n$ and where $\cC_n$ is a cyclic group of order $n$ from now on. 

Let $d$ be a positive integer. The set $\Phi(d)$ of possible torsion structures of elliptic curves defined over number fields of degree $d$ has been deeply studied by several authors. The case $d=1$ was obtained by Mazur \cite{Mazur1,Mazur2}:
$$
\Phi(1) = \left\{ \cC_n \; | \; n=1,\dots,10,12 \right\} \cup \left\{ \cC_2 \times \cC_{2m} \; | \; m=1,\dots,4 \right\}.
$$
The case $d=2$ was completed by Kamienny \cite{Kamienny} and Kenku and Momose \cite{KenkuMomose}. There are not any other cases where $\Phi(d)$ has been completely determined. 

The second author \cite{Najman} has extended this study to the set $\Phi_\Q(d)$ of possible torsion structures over a number field of degree $d$ of an elliptic curve defined over $\Q$. He has obtained a complete description of $\Phi_\Q(2)$ and $\Phi_\Q(3)$. For convenience, we will write here only the latter set:
$$
\Phi_{\Q}(3) = \left\{ \cC_n \; | \; n=1,\dots,10,12,13,14,18,21 \right\} \cup \left\{ \cC_2 \times \cC_{2m} \; | \; m=1\dots,4,7 \right\}.
$$

Fix a possible torsion structure over $\Q$, say $G \in \Phi(1)$. Recently, in \cite{GJT13} the set $\Phi_\Q(2,G)$ of possible torsion structures over a quadratic number field of an elliptic curve defined over $\Q$ such that $E(\Q)_{\tors}\simeq G \in \Phi(1)$ was determined. The first goal of this paper is giving a complete description (see Theorem \ref{main}) of $\Phi_\Q(3,G)$, as was done in \cite[Theorem 2]{GJT13} for the case $d=2$.

Moreover, in \cite{GJT14} the first and third author obtained, for $d=2$ and for all $G\in \Phi(1)$,  the set
$$
\mathcal{H}_{\Q}(d,G) = \{ S_1,...,S_n \}
$$
where, for any $i=1,...,n$, $S_i= \left[ H_1,...,H_m \right]$ is a list, with $H_i \in \Phi_{\mathbb Q}(d,G) \setminus \{ G \}$, and there exists an elliptic curve $E_i$ defined over ${\mathbb Q}$ such that:
\begin{itemize}
\item $E_i({\mathbb Q})_{tors} = G$.
\item There are number fields $K_1,...,K_m$ (non--isomorphic pairwise) of degree $d$ with $E_i \left( K_j \right)_{tors} = H_j$, for all $j=1,...,m$.
\end{itemize}
Note that we are allowing the possibility of two (or more) of the $H_j$ being isomorphic. From these results, it follows \cite{GJT14, Najman2014}:

\begin{corollary}
If $E$ is an elliptic curve defined over $\Q$, then there are at most four quadratic fields $K_i$, $i=1,\dots,4$ (non--isomorphic pairwise), such that $E(K_i)_{\tors}\ne E(\Q)_{\tors}$. That is,
$$
\max_{G \in \Phi(1)} \Big\{ \#S \; \Big| \; S \in \mathcal{H}_{\Q}(2,G) \Big\} = 4.
$$
\end{corollary}

Here, we obtain the equivalent description for the case $d=3$. That is, we give a complete description of $\mathcal{H}_{\Q}(3,G)$ for a given $G \in \Phi(1)$ (see Theorem \ref{hQ3}). Precisely, the main results of this paper are the following:

\begin{theorem}\label{main}
For $G \in \Phi(1)$, the set $\Phi_\Q(3,G)$ is the following:
$$
\begin{array}{|c|c|}
\hline
G & \Phi_\Q \left(3,G \right)\\
\hline
\cC_1 & \left\{ \cC_1\,,\, \cC_2 \,,\, {\cC_3} \,,\, {\cC_4} \,,\, {\cC_6}\,,\, \cC_7\,,\, \cC_{13}\,,\, \cC_2\times \cC_2 \,,\, \cC_2 \times \cC_{14}\, \right\} \\
\hline
\cC_2 & \left\{ \cC_2\,,\, {\cC_6}\,,\, \cC_{14}\, \right\} \\ 
\hline
\cC_3 & \left\{ \cC_3\,,\, \cC_{6}\,,\, \cC_{9}\,,\, {\cC_{12}}\,,\, \cC_{21}\,,\, \cC_{2}\times\cC_{6}\right\} \\
\hline
\cC_4 & \left\{ \cC_4 \,,\, {\cC_{12}}\, \right\} \\
\hline
\cC_5 & \left\{ \cC_5, \; \cC_{10} \right\} \\
\hline
\cC_6 & \left\{ \cC_6\,,\, \cC_{18}\, \right\} \\
\hline
\cC_7 & \left\{ \cC_7 \,,\, \cC_{14}\,\right\} \\
\hline
\cC_8 & \left\{ \cC_8\, \right\} \\
\hline
\cC_9 & \left\{ \cC_9\,,\, \cC_{18}\, \right\} \\
\hline
\cC_{10} & \left\{ \cC_{10}\, \right\} \\
\hline
\cC_{12} & \left\{ \cC_{12}\, \right\} \\
\hline
{\cC_2 \times \cC_2}& \left\{ \cC_2 \times \cC_{2}\,,\, \cC_2 \times \cC_{6}\,\right\} \\
\hline
\cC_2 \times \cC_4 & \left\{ \cC_2 \times \cC_4\, \right\} \\
\hline
\cC_2 \times \cC_6 & \left\{ \cC_2 \times \cC_6\, \right\} \\
\hline
\cC_2 \times \cC_8 & \left\{ \cC_2 \times \cC_8\, \right\} \\
\hline
\end{array}
$$
\end{theorem}

\

\noindent {\bf Remark.} The elements of the sets $\Phi_\Q(3,G)$ were actually found using the computations that can be found in the appendix. These computations also prove that all the listed groups actually are in $\Phi_\Q(3,G)$. The main part of our work has therefore been to prove that there were indeed no more groups in these sets.

\begin{theorem}\label{hQ3}
Let $E$ be an elliptic curve defined over $\Q$. Then:
\begin{itemize}
\item[(i)] There is at most one cubic number field $K$, up to isomorphism, such that 
$$
E(K)_{\tors}\simeq H\ne E(\Q)_{\tors},
$$
for a fixed $H\in \Phi_{\Q}(3)$.
\item[(ii)] There are at most three cubic number fields $K_i$, $i=1,2,3$ (non--isomorphic pairwise), such that 
$$
E(K_i)_{\tors}\ne E(\Q)_{\tors}.
$$
Moreover, the elliptic curve \texttt{162b2} is the unique rational elliptic curve where the torsion grows over three non--isomorphic cubic fields.
\item[(iii)]  Let be $G \in \Phi(1)$ such that $\Phi_\Q \left(3,G \right)\ne \{G\}$. Then the set $\mathcal{H}_{\Q}(3,G)$ consists of the following elements (third row is $h = \#S$, for each $S \in \mathcal{H}_{\Q}(3,G)$):\\[2mm]
\begin{center}
\begin{tabular}{ccc}
\begin{tabular}{|c|l|c|}
\hline
$G$ & $\mathcal{H}_{\Q}(3,G)$ & $h$\\
\hline\hline
\multirow{12}{*}{$\cC_1$} & $\cC_2$  & \multirow{5}{*}{$1$}\\
\cline{2-2}
& {$\cC_4$}   & \\
\cline{2-2}
& {$\cC_6$}   & \\
\cline{2-2}
& {$\cC_2\times \cC_2$}   & \\
\cline{2-2}
& {$\cC_2\times \cC_{14}$}   & \\
\cline{2-3}
& {$\cC_2,\cC_3$}   & \multirow{7}{*}{$2$}\\
\cline{2-2}
& {$\cC_2,\cC_7$ }  & \\
\cline{2-2}
& {$\cC_2,\cC_{13}$}   & \\
\cline{2-2}
& {$\cC_3,\cC_4$}   & \\
\cline{2-2}
& {$\cC_3,\cC_2\times \cC_2$}   & \\
\cline{2-2}
& {$\cC_4,\cC_7$}   & \\
\cline{2-2}
& {$\cC_7,\cC_2\times \cC_2$}   & \\
\cline{2-3}
& {$\cC_2,\cC_3,\cC_7$}   & $3$\\
\hline
\end{tabular}
& \quad \quad 
\begin{tabular}{|c|l|c|}
\hline
$G$ & $\mathcal{H}_{\Q}(3,G)$ & $h$\\
\hline
\hline
\multirow{2}{*}{$\cC_2$} & {$\cC_6$}  & \multirow{2}{*}{$1$}\\
\cline{2-2}
& {$ \cC_{14} $}   & \\
\hline
\multirow{4}{*}{$\cC_3$}  & $ \cC_{6}$   &\multirow{2}{*}{$1$} \\
\cline{2-2}
& {$ \cC_{12} $}   & \\
\cline{2-2}
& {$ \cC_{2}\times\cC_6 $}   & \\
\cline{2-3}
& {$\cC_6,\cC_9$}   & \multirow{2}{*}{$2$} \\
\cline{2-2}
& {$\cC_6,\cC_{21}$}   & \\
\hline
$\cC_4$  & {$ \cC_{12}$}   & $1$\\
\hline
$\cC_5$  & $ \cC_{10}$   & $1$\\
\hline
$\cC_6$  & {$ \cC_{18}$}   & $1$\\
\hline
$\cC_7$  & $ \cC_{14}$   & $1$\\
\hline
$\cC_9$  & $ \cC_{18}$   & $1$\\
\hline
$\cC_{2}\times\cC_2$  & {$\cC_{2}\times\cC_6$}  & $1$\\
\hline
\end{tabular}
\end{tabular}
\end{center}
\end{itemize}
\end{theorem}

The best result previously known \cite[Lemma 3.3]{JKS} stated that the torsion subgroup of a rational elliptic curve grows strictly in only finitely many cubic number fields.

\

\noindent {\bf Notation:} Please note that, in the sequel, for examples and precise curves we will use the Antwerp--Cremona tables and labels \cite{antwerp,cremonaweb}. We will write $G=H$ (respectively $G<H$ or $G \leq H$) for the fact that $G$ is {\em isomorphic} to $H$ (or to  a subgroup of $H$) without further detail on the precise isomorphism.

\section{Auxiliary results}

We will fix once and for all some notations. We will use a short Weierstrass equation for an elliptic curve $E$,
$$
E: \; Y^2 = X^3 + AX + B,  \quad A,B \in \Z,
$$
with discriminant $\Delta$.

For such an elliptic curve $E$ and an integer $n$, let $E[n]$ be the subgroup of all points which order is a divisor of $n$ (over $\overline{\Q}$), and let $E(K)[n]$ be the set of points in $E[n]$ with coordinates in $K$, for a number field $K$. Let us recall the following well-known result \cite[Ch. III, 8.1.1]{Silverman}

\begin{proposition}\label{nthroot}
Let $E$ be an elliptic curve over a number field $K$. If $\cC_m \times \cC_m \le E(K)$, then $K$ contains the cyclotomic field $\Q(\zeta_m)$ generated by the $m$--th roots of unity.
\end{proposition}

Let us fix the set--up, following \cite{Najman}. Let $K/\Q$ be a cubic extension, and $L$ the normal closure of $K$ over $\Q$. Finally, let $M$ be the only subextension $\Q \subset M \subset L$ such that $[L:M]=3$. Therefore, we have two posible situations:

\begin{itemize}
\item The extension $K/\Q$ is Galois. Then $\Q=M$ and $K=L$.
\item The extension $K/\Q$ is not Galois. Then we have 
$$
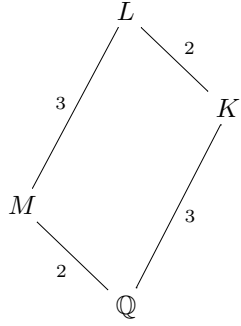
\xymatrix{
 & L\ar@{-}[dr]^2  & \\
  & & K\\
  M\ar@{-}[uur]^3 \ar@{-}[dr]_2  & & \\
  & \mathbb Q \ar@{-}[uur]_3 &
  }
$$
\end{itemize}

\noindent{\bf Remark.} Let $\alpha \in \Q$. If there is some $\beta \in K$ with $\alpha = \beta^2$, then $\beta \in \Q$.

\

Now we will recall some results from \cite{Najman} which will come in handy.

\begin{proposition}\label{najman}
Let $E$ be an elliptic curve defined over $\Q$, $K$, $L$ and $M$ as above, $G\in \Phi_\Q(1)$ and $H\in \Phi_\Q(3)$ such that $E(\Q)_{\tors}\simeq G$ and $E(K)_{\tors}\simeq H$.
\begin{itemize}
\item[(i)] If $G$ has a non-trivial $2$-Sylow subgroup, $G$ and $H$ have the same $2$-Sylow subgroup \cite[Lemma 8]{Najman}.  
\item[(ii)] If $\cC_4\not\leq G$, then  $\cC_8 \not\leq H$ and, if $\cC_4 \leq H$, then $M=\Q(i)$ and $\Delta \in (-1) \cdot (\Q^*)^2$ \cite{dd}, \cite[Corollary 12] {Najman}.
\item[(iii)] $E(K)[5]=E(\Q)[5]$ \cite[Lemma 21]{Najman}.
\item[(iv)] If $H= \cC_{21}$, then $E$ is the elliptic curve \texttt{162b1} and $K=\Q(\zeta_9)^+$ \cite[Theorem 2]{Najman}.
\item[(v)] If $G=\cC_7$ then $H\ne \cC_2\times\cC_{14}$ \cite[Proof Prop. 29]{Najman}.
\item[(vi)] If $E(M)$ has no points of order $3$, neither does $E(L)$ \cite[Lemma 13]{Najman}
\end{itemize}
\end{proposition}

Also some results on isogenies will be needed:

\begin{proposition}\label{isogeny}
Let $E$ be an elliptic curve defined over $\Q$, $K$ and $L$ as above. 
\begin{itemize}
\item[(i)] Assume $E$ has a rational $n$-isogeny. Then either $1\leq n \leq 19$, or $n \in  \{21,25,27,37,43,67,163\}$ \cite{Mazur2,Kenku1,Kenku2,Kenku3}.
\item[(ii)] Assume $n$ is odd and not divisible by $3$. If $E(K)$ has a point of order $n$, then $E$ has a rational isogeny of degree $n$ \cite[Lemma 18]{Najman}.
\item[(iii)] If $F$ is a number field and $E$ has two independent isogenies over $F$ with degrees $n$ and $m$, $E$ is isogeneous (over $F$) to an elliptic curve with an $mn$--isogeny  \cite[Lemma 7]{Najman}.
\item[(iv)] If $K=L$, $n$ is an odd integer and $E(K)$ has a point of order $n$, then $E$ has a rational $n$--isogeny \cite[Lemma 19]{Najman}.
\item[(v)] Let $F$ be a quadratic number field, $n$ an odd integer and $E/\Q$ an elliptic curve such that $\cC_n \leq E(F)$. Then $E$ has a rational $n$--isogeny \cite[Lemma 5]{Najman}.
\item[(vi)] Assume $E(K)$ has a point of order $9$. Then either $E/\Q$ has a $9$--isogeny or it has two independent $3$--isogenies \cite[Proposition 14]{Najman}.
\end{itemize}
\end{proposition}

\begin{lemma}\label{galois}
Let $p$ be prime, $f$ a $p$--isogeny on $E/\Q$, and let $\ker (f)$ be generated by $P$. Then the field of definition $\Q(P)$ of $P$ (and all of its multiples) is a cyclic (Galois) extension of $\Q$ of order dividing $p-1$.
\end{lemma}

\begin{proof}
First note that the fact that $F=\Q(P)$ is Galois over $\Q$ follows immediately from the Galois--invariance of $\langle P \rangle$. Let $\chi$ be the character of the isogeny,
$$
\chi: \mbox{Gal}(F/\Q) \, \longrightarrow \mbox{Aut}(\langle P \rangle).
$$
which, to each element of $\mbox{Gal}(F/\Q)$, adjoins its action on $\langle P \rangle$. It is easy to check that this is a homomorphism. 

Suppose that $\chi$ is not an injection. Then there exists an element $\sigma$, not the identity, such that $\chi(\sigma)= \mbox{id}$, so $\langle \sigma \rangle$ acts trivially on $P$. Denoting $F_0=F^{\sigma}$ (the fixed field of $\langle \sigma \rangle$), every automorphism of $\mbox{Gal}(F/F_0)$ fixes $P$, and hence $P$ is $F_0$--rational, which is in contradiction with the minimality of $F$. 

Since $\mbox{Gal}(F/\Q)$ is isomorphic to a subgroup of $\mbox{Aut}\langle P \rangle$, which is isomorphic to $\cC_{p-1}$, we are finished. 
\end{proof}

\begin{lemma}\label{3iso}
If $E(K)$ has a point of order $3$ over a cubic field $K$, then $E$ has a $3$--isogeny over $\Q$.
\end{lemma}

\begin{proof}
$E(L)$ has a point of order $3$, so $E(M)$ has a point of order 3 from Proposition \ref{najman} (vi). And by Proposition \ref{isogeny} (v), $E$ has a $3$--isogeny over $\Q$.
\end{proof}

\begin{lemma}\label{9-3}
If $E(K)$ has a point of order $9$, then $E(\Q)$ has a point of order $3$. 
\end{lemma}

\begin{proof}
By Proposition \ref{isogeny} (vi) $E/\Q$ has either an isogeny of degree $9$ or $2$ isogenies of degree $3$. 

First suppose it has $2$ isogenies of degree $3$ and no $3$--torsion. Then it follows that $\Q(E[3])$ is a biquadratic field and the intersection of $\Q(E[3])$ and $K$ must be trivial (that is, $\Q$), which contradicts the fact that $E(K)$ has non--trivial $3$--torsion. Hence $E(\Q)$ has a $3$--torsion point.

Now suppose $E/\Q$ has a $9$--isogeny $f$, such that $\ker(f)= \langle P \rangle$, and such that $P$ is $K$--rational. Then the isogeny character 
$$
\chi: \mbox{Gal}(K/\Q) \; \longrightarrow \mbox{Aut}(\langle P \rangle)
$$
sends the generator $\sigma$ of $\mbox{Gal}(K/\Q)$ into an element of order $3$ in $\mbox{Aut}(\langle P \rangle)$, i.e. into [4] or [7]. Both of these act trivially on $\langle 3P \rangle$, implying that $E(\Q)$ has non--trivial $3$--torsion.
\end{proof}

\noindent {\bf Remark.} Now and then we will consider the case where we have $K_1$ and $K_2$ two different cubic number fields. Let us write as usual $K_1K_2$ for the compositum field of both extensions. Then one of these two situations hold:

\begin{itemize}
\item $[K_1K_2:\Q] = 9$.
\item $[K_1K_2:\Q] = 6$. In this case, $K_1$ and $K_2$ are isomorphic and $K_1K_2$ is the Galois closure of both fields over $\Q$.
\end{itemize}

\section{Proof of Theorem \ref{main}}

Note that from Proposition \ref{najman} (i), if $G=\cC_{2n}$, for some $n\ne 0$, then $\cC_2\times\cC_{2}\not\subset H$. 

Also from Proposition \ref{najman} (i) and the description of $\Phi_\Q(3)$, we can solve the non--cyclic cases from Theorem \ref{main} easily, as we know that
$$
\Phi_\Q(3,\cC_2 \times \cC_{2n}) \leq
\left\{
\begin{array}{lcl}
\left\{ \cC_2 \times \cC_{2}, \; \cC_2 \times \cC_6, \; \cC_2 \times \cC_{14}  \right\} & & \mbox{if $n=1$},\\ 
\left\{ \cC_2 \times \cC_{2n}  \right\}  & & \mbox{if $n\ne 1$}.
\end{array}
\right.
$$

The only case that will not happen and we cannot discard already is $G = \cC_2 \times \cC_2$, $H=\cC_2 \times \cC_{14}$. But this case cannot happen as, from Proposition \ref{isogeny} (ii) and (iii), that would imply $E$ has a $28$--isogeny, contradicting Proposition \ref{isogeny} (i). This finishes the non--cyclic case.

Let us move therefore to the cyclic case. The groups $H$ from $\Phi_\Q(3)$ that do not appear in some $\Phi_\Q(3,G)$, with a $G<H$ and $G$ cyclic can be ruled out from $\Phi_\Q(3,G)$ most of the times using the previous results. In the table below we indicate:
\begin{itemize}
\item With (i) - (vi), which part of Proposition \ref{najman} is used, 
\item With ({\bf \ref{9-3}}), the case is ruled out from Lemma \ref{9-3},
\item With $-$, the case is ruled out because $G \not\subset H$, 
\item With $\checkmark$, the case is possible (and in fact, it occurs).
\end{itemize}

The table (row$=H$, column$=G$) deals with the case $G$ cyclic. 

\vspace{.3cm}

\begin{center}
\begin{tabular}{|c|c|c|c|c|c|c|c|c|c|c|c|}
\cline{2-12}
\multicolumn{1}{c|}{} & $\cC_1$ & $\cC_2$ & $\cC_3$ & $\cC_4$ & $\cC_5$ & $\cC_6$ & $\cC_7$ & $\cC_8$ & $\cC_9$ & $\cC_{10}$ & $\cC_{12}$ \\
\hline
$\cC_1$ & $\checkmark$ & $-$ & $-$ & $-$ & $-$ & $-$ & $-$ & $-$ & $-$ & $-$ & $-$ \\
\hline
$\cC_2$ & $\checkmark$ & $\checkmark$ & $-$ & $-$ & $-$ & $-$ & $-$ & $-$ & $-$ & $-$ & $-$  \\
\hline
$\cC_3$ & $\checkmark$ & $-$ & $\checkmark$ & $-$ & $-$ & $-$ & $-$ & $-$ & $-$ & $-$ & $-$  \\
\hline
$\cC_4$ &$\checkmark$  & (i) & $-$ & $\checkmark$ & $-$ & $-$ & $-$ & $-$ & $-$ & $-$ & $-$ \\
\hline
$\cC_5$ & (iii) & $-$ & $-$ & $-$ & $\checkmark$ & $-$ & $-$ & $-$ & $-$ & $-$ & $-$  \\
\hline
$\cC_6$ & $\checkmark$ & $\checkmark$ &$\checkmark$  & $-$ & $-$ & $\checkmark$ & $-$ & $-$ & $-$ & $-$ & $-$ \\
\hline
$\cC_7$ & $\checkmark$ & $-$ & $-$ & $-$ & $-$ & $-$ & $\checkmark$ & $-$ & $-$ & $-$ & $-$ \\
\hline
$\cC_8$ & (ii) & (i) & $-$ & (i) & $-$ & $-$ & $-$ & $\checkmark$ & $-$ & $-$ & $-$  \\
\hline
$\cC_9$ & ({\bf \ref{9-3}}) & $-$ & $\checkmark$ & $-$ & $-$ & $-$ & $-$ & $-$ & $\checkmark$ & $-$ & $-$\\
\hline
$\cC_{10}$ & (iii) & (iii) & $-$ & $-$ & $\checkmark$  & $-$ & $-$ & $-$ & $-$ & $\checkmark$ & $-$  \\
\hline
$\cC_{12}$ & {(?)} & (i) & $\checkmark$& $\checkmark$ & $-$ & (i) & $-$ & $-$ & $-$ & $-$ & $\checkmark$ \\
\hline
$\cC_{13}$ & $\checkmark$ & $-$ & $-$ & $-$ & $-$ & $-$ & $-$ & $-$ & $-$ & $-$ & $-$ \\
\hline
$\cC_{14}$ & {(?)} & $\checkmark$ & $-$ & $-$ & $-$ & $-$ & $\checkmark$ & $-$ & $-$ & $-$ & $-$  \\
\hline
$\cC_{18}$ & ({\bf \ref{9-3}}) & ({\bf \ref{9-3}}) &  {(?)} & $-$ & $-$ & $\checkmark$ & $-$ & $-$ & $\checkmark$ & $-$ & $-$  \\
\hline
$\cC_{21}$ & (iv) & $-$ & $\checkmark$ & $-$ & $-$ & $-$ & $-$ & $-$ & $-$ & $-$ & $-$  \\
\hline
$\cC_2 \times \cC_2$ & $\checkmark$ & (i) & $-$ & $-$ & $-$ & $-$ & $-$ & $-$ & $-$ & $-$ & $-$  \\
\hline
$\cC_2 \times \cC_4$ & {(?)} &   (i)  &  $-$  & (i) & $-$ & $-$ & $-$ & $-$ & $-$ & $-$ & $-$ \\
\hline
$\cC_2 \times \cC_6$ & {(?)} &  (i) & $-$ & $-$ & $-$ & (i) & $-$ & $-$ & $-$ & $-$ & $-$ \\
\hline
$\cC_2 \times \cC_8$ & (ii) & (i)  & $-$ & (i) & $-$ & $-$ & $-$ & $(i)$ & $-$ & $-$ & $-$ \\
\hline
$\cC_2 \times \cC_{14}$ & $\checkmark$ &  (i)  & $-$ & $-$ & $-$ & $-$ & (v) & $-$ & $-$ & $-$ & $-$ \\
\hline
\end{tabular}-
\end{center}

\vspace{.3cm}

Let us now discard the remaining cases.

\subsection*{The case $G = \cC_1$, $H = \cC_{12}$.}

In this case, from Proposition \ref{najman} (ii,vi), we already know that $M = \Q(i)$ and $E(M)[3] \neq \{ \cO \}$. Again as above, having points of order $3$ in both $M$ and $K$ implies that these are independent points and hence $E[3](L)\simeq \cC_3 \times \cC_3$, from which it follows that $M = \Q(\zeta_3)$, which is a contradiction.

\subsection*{The case $G = \cC_1$, $H = \cC_{14}$.} In this case $E$ must have a rational $7$--isogeny, from Proposition \ref{isogeny} (ii). Then, from Lemma \ref{galois} we know that $K$ is a cyclic cubic Galois extension, hence $K=L$. Under these circumstances, $E(K)[2]$ cannot be $\cC_2$, as $K$ is  either the splitting field of $X^3+AX+B$ (in which case $E(K)[2] = \cC_2 \times \cC_2$) or is irreducible over $K$, in which case there are no points of order $2$ in $E(K)$.

\subsection*{The case $G = \cC_1$, $H = \cC_2 \times \cC_4$.}

Assume our curve is given in Weierstrass short form
$$
Y^2 = X^3+AX+B.
$$

If $G$ is cyclic and $H$ is not, $K$ must be the splitting field of $X^3+AX+B$. So in this case $\Q=M$, and $K=L$, but this contradicts Proposition \ref{najman} (ii).

\subsection*{The case $G = \cC_1$, $H = \cC_2 \times \cC_6$.}

As in the previous case, $\Q=M$, and $K=L$. But there are points of order $3$ in $E(L)$, so $E(M)[3] \neq \{ \cO \}$, but this contradicts $G= \cC_1$, as $\Q=M$.

\subsection*{The case $G = \cC_3$, $H = \cC_{18}$.} As we gain exactly one $2$--torsion point in the passing from $\Q$ to $K$, we already know that $K$ is not Galois and, in fact, $L$ must be the splitting field of $X^3+AX+B$. Then, from Lemma \ref{galois} and Proposition \ref{isogeny} (vi) we have that $E(\Q)$ must have $2$ isogenies of degree $3$.

Now we look at how $\mbox{Gal}(L/\Q)$ acts on $E[9]$. The $L$--rational points have to be sent to $L$--rational points. So if $P$ is an $L$--rational point of order $9$, the generators of $\mbox{Gal}(L/\Q)$ cannot both send $P$ to a multiple of $P$, because this would imply that $\langle P \rangle$ is $\mbox{Gal}(L/\Q)$--invariant (and hence $\mbox{Gal}(\bar \Q/\Q)$--invariant), which would imply a $9$--isogeny over $\Q$. So this means that $E[9](L)$ is strictly larger than $C_9$. The only possibility is that $E[9](L)=C_3 \times C_9$ and this implies $M=\Q(\sqrt{-3})$ because of Proposition \ref{nthroot}.

As $L$ is the splitting field of $X^3+AX+B$, this really implies $E(L)_{tors} \leq C_6 \times C_{18}$. Moreover, as the quadratic subextension of $L$ is $\Q(\sqrt{-3})$, $L$ is a pure cubic field and our curve is a Mordell curve $Y^2=X^3+n$, for some $n \in \Z$. But the only elliptic curve with $j$--invariant $0$ defined over $\Q$ which has full $3$--torsion over $\Q(\sqrt{-3})$ is \texttt{27a1} (and also its $-3$ twist), and by simply computing that this curve has $L$--torsion $C_6 \times C_6$, we are finished.

\section{Proof of Theorem \ref{hQ3}}

\subsection*{Proof of (i)}
Let $E$ be an elliptic curve defined over $\Q$  such that $E(\Q)_{\tors} \simeq G\in \Phi(1)$ and $H\in \Phi_{\Q}(3)$. Let us prove that there is at most one cubic number field $K$ such that $E(K)_{\tors}\simeq H\ne G$.

First, let be $H=G\times \cC_m$  such that  $\gcd(|G|,m)=1$. Suppose that there exist two cubic fields $K_1$ and $K_2$ such that  $E(K_i)_{\tors}\simeq H$, $i=1,2$. Then $\cC_m\times \cC_m\le E(L)_{\tors}$,  where $L$ is the degree $9$ number field obtained by composition of $K_1$ and $K_2$. Therefore, $\Q(\zeta_m)\subset L$, which implies that $\varphi(m)$ divides $9$. This eliminates the following possibilities: 
\begin{itemize}
\item $G=\cC_1$ and $H\in\{\cC_3,\cC_4,\cC_6,\cC_7,\cC_{13}\}$; 
\item $G=\cC_2$ and $H\in\{\cC_6, \cC_{14}\}$;
\item $G=\cC_3$ and $H\in\{\cC_{12}, \cC_{21}\}$;
\item $G=\cC_4$ and $H=\cC_{12}$;
\item $G= \cC_2\times\cC_2$ and $H=\cC_2\times\cC_6$;
\end{itemize}

On the other  hand, if the order of $G$ is odd then there is at most one $H$ of even order with $G<H$. The cubic field is the one defined by the $2$--division polynomial of the elliptic curve. This argument therefore crosses out the cases:
\begin{itemize}
\item $G=\cC_1$ and $H\in\{\cC_2, \cC_2\times\cC_2,\cC_2\times\cC_{14}\}$; 
\item $G=\cC_3$ and $H\in\{\cC_6, \cC_2\times\cC_6\}$;
\item $G=\cC_5$ and $H=\cC_{10}$;
\item $G=\cC_7$ and $H=\cC_{14}$;
\item $G=\cC_9$ and $H=\cC_{18}$;
\end{itemize}

The remaining cases to be dealt with are $G= \cC_3$ with $H=\cC_{9}$ and $G= \cC_6$ with $H=\cC_{18}$. These are essentialy the same since $\cC_6=\cC_2\times\cC_3$ and $\cC_{18}=\cC_2\times\cC_9$. Assume we have $\langle P\rangle\simeq \cC_9$, $\langle Q\rangle\simeq \cC_9$, where $P$ and $Q$ are defined over two non-isomorphic cubic fields. Therefore $P$ is not a multiple of $Q$ and $Q$ is not a multiple of $P$ and $\cC_3\times\cC_3\leq \langle P,Q\rangle$. This is imposible, since both $P$ and $Q$ would be defined over a field of degree $9$, which cannot contain $\Q(\zeta_3)$.

This proves the first statement of Theorem \ref{hQ3}.

\subsection*{Proof of (ii) and (iii)}
First note that if 
$$
E:Y^2=f(X)
$$ 
is an elliptic curve defined over $\Q$ such that $E(\Q)_{\tors} \simeq G$ has odd order, then $f(X)$ is an irreducible cubic polynomial. Now, denote by $K$ the cubic field defined by $f(X)$, then $H= E(K)_{\tors}$ satisfies that $G\ne H$ and $H$ is of even order. Moreover, $H$ is the unique group of even order such that $H\in S$, for any $S \in \mathcal{H}_{\Q}(3,G)$ because $f(X)$ is the $2$-division polynomial of $E$.

\vspace{1mm}

Now, for any $G\in\Phi(1)$ let us construct the elements $S \in \mathcal H_\Q(3,G)$ in ascending order of $\#S$. In Table \ref{tablagrande} (see Appendix) we show examples for all the possible cases of $S$ (after taking into account the preliminary remark) for any $G\in\Phi(1)$. Now, by (i) we know that there are not repeated elements in any $S \in \mathcal H_\Q(3,G)$. Then the possible cases with $\#S>1$ come from $G=\cC_1, \cC_2, \cC_3$: 

\vspace{1mm}

\noindent \fbox{$G=\cC_1$} 

\vspace{1mm}

We have examples in Table \ref{tablagrande} for any $S\in  \mathcal H_\Q(3,\cC_1)$ with $\#S=2$ except for the cases:
$$
[\cC_4,\cC_{13}],\; [\cC_3,\cC_6],\; [\cC_6,\cC_7],\; [\cC_6,\cC_{13}],
$$
$$
[\cC_2\times\cC_2,\cC_{13}],\; [\cC_2\times\cC_{14},\cC_3],\; [\cC_2\times\cC_{14},\cC_7],\; [\cC_2\times\cC_{14},\cC_{13}].
$$

\begin{itemize}
\item As for $[\cC_4,\cC_{13}]$, if such a curve existed then it would have to have discriminant $-Y^2$ (as it gains $4$--torsion - see Proposition \ref{najman} (ii)) for some rational $Y$. On the other hand, the curve must have a $13$--isogeny over $\Q$, which implies its discriminant is of the form \cite[Lemma 27]{Najman}
$$
\Delta = \Box \cdot t(t^2+6t+13)
$$
where $\Box$ is a rational square. Therefore such a curve would give a rational non--trivial (meaning $Y\neq 0$) solution of the equation 
$$
Y^2= X^3-6X^2+13X,
$$
but one easily checks that there are none.

\item Looking at $[\cC_3,\cC_6]$ we find that $E$ gains full $3$--torsion over the compositum of two cubic extensions, $K_1$ and $K_2$, because the fields cannot be isomorphic, hence the points of order $3$ in $K_1$ and $K_2$ are independent. This implies $\Q(\zeta_3) \subset K_1K_2$, which is impossible as $[K_1K_2:\Q]=9$ in this case.

\item Let us look at the pair $[\cC_6,\cC_7]$. The existence of $\cC_6$ implies a $3$--isogeny over $\Q$ and the existence of $\cC_7$ implies a rational $7$--isogeny, hence $E$ has a $21$--isogeny. Therefore $E$ is a twist of an elliptic curve in the \texttt{162b} isogeny class. It can be seen that only one elliptic curve in each of the $4$ family of twists gains $7$--torsion in a cubic extension. Thus there are in fact $4$ curves that we need to check, all in all. For each of the $4$ curves we can check whether the curve gains any $3$--torsion in the fields where it gains $2$--torsion, and discard all the cases. 

\item The case $[\cC_6,\cC_{13}]$ can be ruled out as, from Proposition \ref{isogeny} (iii) and Lemma \ref{3iso}, it would imply the existence of a curve with a rational $39$--isogeny, contradicting Proposition \ref{isogeny} (i).
\item The case $[\cC_2 \times \cC_2, \cC_{13}]$ is very similar to the first one, the only difference being that, gaining full $2$--torsion over a cubic field, the discriminant must be a square. Anyway, the corresponding equation
$$
Y^2= X^3+6X^2+13X,
$$
still has no solutions with $Y\neq 0$.

\item Let us look at the case $[\cC_2\times\cC_{14},\cC_3]$. A curve featuring these torsion extensions would have a $21$--isogeny from Proposition \ref{isogeny} (ii,iv) and Lemma \ref{3iso} and also would gain full $2$--torsion over a cubic field, so as in the previous case its discriminant must be a square. But the elliptic curves with a $21$--isogeny have discriminant $-2 \cdot \Box$, where $\Box$ is a rational square \cite[pp. 78--80]{antwerp}. Hence this case is not possible.

\item We can remove the  case $[\cC_2\times\cC_{14},\cC_{7}]$, similarly as the second case. In this case we would have two cubic extensions $K_1$ and $K_2$ which must verify $[K_1K_2:\Q]=9$, as $X^3+AX+B$ splits completely in one of them and remains irreducible in the other. As $\Q(\zeta_7)\subset K_1K_2$ using Proposition \ref{nthroot} above, we reach a contradiction. 

\item The last case, that of $[\cC_2\times\cC_{14},\cC_{13}]$, is also removable as it would similarly imply the existence of a rational elliptic curve with a $91$--isogeny.
\end{itemize}

Now, we need to prove that the only $S\in\mathcal H_\Q(3,\cC_1)$ with $\#S=3$ is  $[\cC_2,\cC_{3},\cC_7]$. For this purpose
we have to remove the cases:
$$
[\cC_2,\cC_3,\cC_{13}],\;[\cC_2,\cC_7,\cC_{13}],\;[\cC_3,\cC_4,\cC_{7}],\; [\cC_2\times\cC_2,\cC_{3},\cC_7].
$$

\begin{itemize}
\item The first case can be ruled out as $[\cC_6,\cC_{13}]$ above, for it implies the existence of a rational curve with a $39$--isogeny.
\item The second case, as $[\cC_2\times\cC_{14},\cC_{13}]$ above, would imply the existence of a rational elliptic curve with a $91$--isogeny. Hence it cannot happen.
\item The third case is eliminated by noting that the discriminant of such a curve should be $-Y^2$ (for it gains $4$--torsion) and $-2 \cdot \Box$, where $\Box$ is a rational square (for it has a $21$--isogeny).
\item The last case is similar to the case $[\cC_2\times\cC_{14},\cC_3]$ above.
\end{itemize}

Looking with greater detail at the case $[\cC_2,\cC_{3},\cC_7]$ we find that if a curve gains torsion in such a way in three non--isomorphic cubic fields, it must have a $21$--isogeny and in fact (as in the $[\cC_6,\cC_7]$ case) it can only be a very precise curve a family of twists in the \texttt{162b} isogeny class. There are only $4$ such curves and \texttt{162b2} is the only one that grows strictly in three cubic extensions.

\vspace{1mm}

\noindent \fbox{$G=\cC_2$}

\vspace{1mm}

The only case to discard here is $[\cC_6, \, \cC_{14}]$. If such a curve (say $E$) existed, it would follow that $E$ would have a $3$--isogeny and $7$--isogeny and hence a $21$--isogeny. $E$ would also have to contain $\cC_2$, since the odd isogeny cannot kill this torsion. But there do not exist elliptic curves with $21$--isogenies and non--trivial $2$--torsion over $\Q$ \cite[pp. 78--80]{antwerp}.

\vspace{1mm}

\noindent \fbox{$G=\cC_3$}

\vspace{1mm}

We have examples in Table \ref{tablagrande} for any $S\in  \mathcal H_\Q(3,\cC_3)$ with $\#S=2$ except for the cases:
$$
\left[ \cC_9,\cC_{12} \right], \; \left[ \cC_{12}, \cC_{21} \right], \; \left[\cC_2 \times \cC_6,\cC_9 \right], \; 
 \left[ \cC_2 \times \cC_6, \cC_{21} \right]
 $$

\begin{itemize}
\item $[\cC_9,\cC_{12}]$. From Proposition \ref{isogeny} (vi) our curve has either a $9$--isogeny or two independent $3$--isogenies and $\Q(E[3]) = \Q(\zeta_3)$. Moreover from Proposition \ref{najman} (iii) $\Delta \in (-1) \cdot (\Q^*)^2$.

Assume that $E$ has two independent $3$--isogenies and $\Q(E[3]) = \Q(\zeta_3)$. From \cite[p. 147]{Paladino} we get\footnote{Note there is a misprint in the original article, $h^4$ in the numerator should be replaced by $h^6$.}
$$
\Delta = -216 \frac{b^3(h^6-6h^2b^2+12b^3)}{h^6}, \quad b,h \in \Q.
$$

As $\Delta = -y^2$ for some $y \in \Q$, the existence of $E$ implies there are $b,h,y \in \Q$ with 
$$
\left( \frac{y}{bh} \right)^2 = 6 \left( \frac{b}{h^2} \right) \left[ 1-6 \left( \frac{b}{h^2} \right)^2-12 \left( \frac{b}{h^2}
\right)^3 \right],
$$
that is a rational point on the curve 
$$
Y^2=6X \left( 1-6X^2-12X^3 \right),
$$
but its Mordell--Weil group is trivial, and the trivial point do not yield an elliptic curve $E$.

So we are bound to assume $E$ has a $9$--isogeny. From \cite[Appendix]{Ingram}, it follows that $E$ is a twist of $u^2=v^3+av+b$, where
$$
a= -3x(x^3-24), \quad b= 2(x^6-36x^6+216),
$$
for some $x \in \Q$. Then the discriminat of this curve is 
$$
2^{12} 3^6 (c^3-27) u^{12},
$$
where the twelfth power may appear because of twisting. As this should be in $(-1) \cdot (\Q^*)^2$, it should give a point on
$$
Y^2= X^3-27.
$$

The points in this curve can be easily computed (we have done it with \verb|Magma| \cite{magma}); there is only the point at infinity and a point of order 2 that discriminant $0$, so we are done.

\item Second and fourth cases are not possible, as the only curve whose torsion grows to $\cC_{21}$ is \texttt{162b1}, and this curve fits neither of these cases (see Table \ref{tablagrande}). 

\item $[\cC_2 \times \cC_6,\cC_9]$. This case parallels the first one. The only formal change is that, as we gain full $2$--torsion in a cubic extension, $\Delta \in (\Q^*)^2$. Hence, the same arguments lead us to state that such a curve must yield either a point on
$$
Y^2=-6X \left( 1-6X^2-12X^3 \right),
$$
if it has two independent rational $3$--isogenies, or a point on 
$$
Y^2= X^3+27,
$$
should it have a rational $9$--isogeny. As both cases can be checked to be impossible, we are finished.
\end{itemize}

Finally, we see that there are no $S\in\mathcal H_\Q(3,\cC_3)$ with $\#S=3$. Such $S$ should have two groups of odd order. These must be $\cC_{9}$ and $\cC_{21}$. But again the unique elliptic curve over $\Q$ with $\cC_{21}$ over a cubic field is \texttt{162b1} and for this curve, this is not the case (see Table \ref{tablagrande}).

\section*{Appendix: Computations}

Let $G \in \Phi(1)$, $S= \left[ H_1,...,H_m \right]\in\mathcal{H}_{\Q}(3,G)$, $E$ an elliptic curve defined over $\Q$ such that $E({\mathbb Q})_{\tors} = G$ and let $K_1,\dots,K_m$ cubic fields, such that 
$$
E( K_i)_{tors} = H_i \mbox{ for } i=1,...,m.
$$
Table \ref{tablagrande} shows an example of every possible situation, where
\begin{itemize}
\item the first column is $G$,
\item the second column is $S \in\mathcal{H}_{\Q}(3,G)$,
\item the third column is $\# S$,
\item the fourth column is the label of the elliptic curve $E$ with minimal conductor satisfying the conditions above, 
\item the fifth column displays a defining cubic polynomial corresponding to the respective $K_i$ of $H_i$ in $S$,
\item the sixth column displays the discriminant of the corresponding $K_i$.
\end{itemize}

\begin{table}[ht]
\caption{$h= \# S \mbox{ for } S \in \mathcal{H}_{\Q}(3,G)$}\label{tablagrande}
\begin{tabular}{|c|l|c||c|c|c|}
\hline
$G$ & $\mathcal{H}_{\Q}(3,G)$ & $h$ & label & cubics & $\Delta$\\
\hline\hline
\multirow{22}{*}{$\cC_1$} & $\cC_2$  & \multirow{5}{*}{$1$} & \texttt{11a2} & $x^3 - x^2 + x + 1$ & $-44$\\
\cline{2-2}\cline{4-6}
& {$\cC_4$}   & & \texttt{338b2} & $x^3 - x^2 - 4x + 12$ & $-676$\\
\cline{2-2}\cline{4-6}
& {$\cC_6$}   & & \texttt{108a2} & $x^3 - 2$ & $-108$\\
\cline{2-2}\cline{4-6}
& {$\cC_2\times \cC_2$}   & & \texttt{196a1} & $x^3 - x^2 - 2 x + 1$ & $49$\\
\cline{2-2}\cline{4-6}
& {$\cC_2\times \cC_{14}$}   & & \texttt{1922c1} & $x^3 - x^2 - 10x + 8$ & $961$\\
\cline{2-3}\cline{4-6}
& {$\cC_2,\cC_3$}   & \multirow{14}{*}{$2$}& \texttt{19a2} & $\begin{array}{c}
x^3 - 2x - 2\\
x^3 - x^2 - 6x - 12\end{array}$ & $\begin{array}{c}
-76\\
-1083\end{array}$\\
\cline{2-2}\cline{4-6}
& {$\cC_2,\cC_7$ }  & & \texttt{294a1} & $\begin{array}{c}
x^3 - x^2 - 2x - 6\\
x^3 - x^2 - 2x + 1\end{array}$ &
 $\begin{array}{c}
-1176\\
49\end{array}$
\\
\cline{2-2}\cline{4-6}
& {$\cC_2,\cC_{13}$}   & & \texttt{147b1} & $\begin{array}{c}
x^3 - x^2 + 5x + 1\\
x^3 - x^2 - 2x + 1\end{array}$ & $\begin{array}{c}
-588\\
49\end{array}$\\
\cline{2-2}\cline{4-6}
& {$\cC_3,\cC_4$}   & & \texttt{162d2} & $\begin{array}{c}
x^3 - 2\\
x^3 - 3x - 4\end{array}$ &
$\begin{array}{c}
-108 \\ -324
\end{array}$
\\
\cline{2-2}\cline{4-6}
& {$\cC_3,\cC_2\times \cC_2$}   & & \texttt{196b2} & $\begin{array}{c}
x^3 - x^2 + 5x + 1\\
x^3 - x^2 - 2x + 1\end{array}$ &
$\begin{array}{c}
-588\\ 49
\end{array}$
\\
\cline{2-2}\cline{4-6}
& {$\cC_4,\cC_7$}   & & \texttt{338b1} & $\begin{array}{c}
x^3 - x^2 - 4x + 12\\
x^3 - x^2 - 4x - 1\end{array}$ &
$\begin{array}{c}
-676\\169
\end{array}$
\\
\cline{2-2}\cline{4-6}
& {$\cC_7,\cC_2\times \cC_2$}   & & \texttt{3969a1} & $\begin{array}{c}
x^3 - 21x - 35\\
x^3 - 21x - 28\end{array}$ &
$\begin{array}{c}
3969\\ 3969
\end{array}$
\\
\cline{2-3}\cline{4-6}
& {$\cC_2,\cC_3,\cC_7$}   & $3$& \texttt{162b2} & 
$
\begin{array}{c}
x^3 - 3x - 10\\
x^3 - 2\\
x^3 - 3x - 1
\end{array}$ &
$\begin{array}{c}
-648\\ -108 \\ 81
\end{array}$

\\
\hline
\hline
\multirow{2}{*}{$\cC_2$} & {$\cC_6$}  & \multirow{2}{*}{$1$}& \texttt{14a3} & $x^3 - 7$ & $-1323$\\
\cline{2-2}\cline{4-6}
& {$ \cC_{14} $}   & & \texttt{49a3} & $x^3 - x^2 - 2x + 1$ & $49$\\
\hline
\hline
\multirow{7}{*}{$\cC_3$}  & $ \cC_{6}$   &\multirow{3}{*}{$1$} & \texttt{19a1} & $x^3 - 2x - 2$ & $ -76$\\
\cline{2-2}\cline{4-6}
& {$ \cC_{12} $}   & & \texttt{162d1} & $x^3 - 3x - 4$ & $ -324$\\
\cline{2-2}\cline{4-6}
& {$ \cC_{2}\times\cC_6 $}   & & \texttt{196b1} & $x^3 - x^2 - 2x + 1 $ & $49$\\
\cline{2-3}\cline{4-6}
& {$\cC_6,\cC_9$}   & \multirow{4}{*}{$2$} & \texttt{19a3} & $\begin{array}{c}
x^3 - 2x - 2\\
x^3 - x^2 - 6x + 7\end{array}$ & $\begin{array}{c}
-76 \\ 361
\end{array}$\\
\cline{2-2}\cline{4-6}
& {$\cC_6,\cC_{21}$}   & & \texttt{162b1} & $\begin{array}{c}
x^3 - 3x - 10\\
x^3 - 3x - 1\end{array}$ & $\begin{array}{c}
-648 \\ 81
\end{array}$\\
\hline
\hline
$\cC_4$  & {$ \cC_{12}$}   & $1$& \texttt{90c1} & $x^3 - x^2 - 3x - 3 $ & $-300$\\
\hline
\hline
$\cC_5$  & $ \cC_{10}$   & $1$& \texttt{11a1} & $x^3 - x^2 + x + 1$ & $ -44$\\
\hline
\hline
$\cC_6$  & {$ \cC_{18}$}   & $1$& \texttt{14a4} & $x^3 - x^2 - 2x + 1$&$49$\\
\hline
\hline
$\cC_7$  & $ \cC_{14}$   & $1$& \texttt{26b1} & $x^3 - x - 2$ & $-104$\\
\hline
\hline
$\cC_8$  &    & $0$& \texttt{} & $$ & \\
\hline
\hline
$\cC_9$  & $ \cC_{18}$   & $1$& \texttt{54b3} & $x^3 + 3x - 2$ & $-216$\\
\hline
\hline
$\cC_{10}$  &   & $0$& \texttt{} & $$ & \\
\hline
\hline
$\cC_{12}$  &   & $0$& \texttt{} & $$ & \\
\hline
\hline
$\cC_{2}\times\cC_2$  & {$\cC_{2}\times\cC_6$}  & $1$& \texttt{30a6} & $x^3 - 3$ & $-243$\\
\hline
\hline
$\cC_{2}\times\cC_4$  &   & $0$& \texttt{} & $$ & \\
\hline
\hline
$\cC_{2}\times\cC_6$  &   & $0$& \texttt{} & $$ & \\
\hline
\hline
$\cC_{2}\times\cC_8$  &   & $0$& \texttt{} & $$ & \\
\hline\end{tabular}
\end{table}

\end{document}